\documentclass[a4paper]{amsart}

\usepackage[english]{babel} 
\usepackage{amsmath,amsthm,amsfonts,amssymb,graphicx,units,color}

\newcommand{\comment}[1]{}

\theoremstyle{plain}
\newtheorem{theo}{Theorem}[section]
\newtheorem{lem}[theo]{Lemma}

\newtheorem{cor}[theo]{Corollary}
\newtheorem{rem}[theo]{Remark}

\theoremstyle{definition}

\numberwithin{equation}{section}

\def\reals{\mathbb{R}}

\def\cf{c_{\mathsf{f}}}

\def\p{\partial}

\def\d{{\rm d}}

\def\dlo{{\rm d}\lambda_{1}}

\def\dld{{\rm d}\lambda_{d}}

\newcommand{\brac}[1]{\left(#1\right)}

\def\ol{\overline}

\newcommand{\set}[2]{\{#1\,\mid\,#2\}}

\DeclareMathOperator{\A}{A}

\def\na{\nabla}

\def\bna{\na_{\!\!\times}}

\DeclareMathOperator{\opdiv}{div}
\def\div{\opdiv}
\def\bdiv{\div_{\!\times}}

\DeclareMathOperator{\sop}{\mathcal{L}}

\def\om{\Omega}

\def\qt{\Xi}
\def\st{I\times\p\om}

\def\ft{\tilde{f}}
\def\ut{\tilde{u}}
\def\pt{\tilde{p}}

\DeclareMathOperator{\hilbert}{\sf H}
\DeclareMathOperator{\cont}{\sf C}
\DeclareMathOperator{\lebesgue}{\sf L}

\DeclareMathOperator{\divergence}{\sf D}

\def\ciga{\mathring{\cont}{}^{\infty}}
\def\cigaom{\ciga(\om)}

\newcommand{\lgen}[2]{\lebesgue^{#1}_{#2}}
\def\lt{\lgen{2}{}}
\def\ltom{\lt(\om)}
\def\ltqt{\lt(\qt)}

\newcommand{\hgen}[2]{\hilbert^{#1}_{#2}}
\newcommand{\hgenz}[2]{\mathring{\hilbert}{}^{#1}_{#2}}

\def\ho{\hgen{1}{}}
\def\hoom{\ho(\om)}
\def\hoga{\hgenz{1}{}}
\def\hogaom{\hoga(\om)}

\def\bho{\hgen{1}{\!\times}}
\def\bhoom{\bho(\om)}

\def\V{\hgen{1,\Delta}{}}
\def\Vom{\V(\om)}
\def\Vga{\hgenz{1,\Delta}{}}
\def\Vgaom{\Vga(\om)}

\def\hzo{\hgen{0;1}{}}
\def\hzoqt{\hzo(\qt)}
\def\hzost{\hgenz{0;1}{}}
\def\hzostqt{\hzost(\qt)}

\def\hoz{\hgen{1;0}{}}
\def\hozqt{\hoz(\qt)}

\def\hoo{\hgen{1;1}{}}
\def\hooqt{\hoo(\qt)}
\def\hoost{\hgenz{1;1}{}}
\def\hoostqt{\hoost(\qt)}

\def\W{\hgen{1;1,\Delta}{}}
\def\Wqt{\W(\qt)}
\def\Wst{\hgenz{1;1,\Delta}{}}
\def\Wstqt{\Wst(\qt)}

\newcommand{\dgen}[2]{\divergence^{#1}_{#2}}
\def\d{\dgen{}{}}
\def\dz{\dgen{}{0}}
\def\dom{\d(\om)}
\def\dzom{\dz(\om)}
\def\dqt{\d(\qt)}

\def\bd{\dgen{}{\!\times}}
\def\bdom{\bd(\om)}

\newcommand{\norm}[1]{|#1|}
\newcommand{\normltom}[1]{|#1|_{\ltom}}
\newcommand{\normltqt}[1]{|#1|_{\ltqt}}

\newcommand{\scp}[2]{\langle#1,#2\rangle}
\newcommand{\scpltom}[2]{\scp{#1}{#2}_{\ltom}}
\newcommand{\scpltqt}[2]{\scp{#1}{#2}_{\ltqt}}

\title[\sc Functional A Posteriori Error Equalities for Linear PDEs]
{\Large\sf An Elementary Method of Deriving A Posteriori Error Equalities\\
and Estimates for Linear Partial Differential Equations}
\author{Immanuel Anjam}
\author{Dirk Pauly}
\address{Fakult\"at f\"ur Mathematik,
Universit\"at Duisburg-Essen, Campus Essen, Germany}
\email[Immanuel Anjam]{immanuel.anjam@uni-due.de}
\email[Dirk Pauly]{dirk.pauly@uni-due.de}
\date{\today}

\begin{document}


\begin{abstract}
In this paper we present a simple method of deriving a posteriori error equalities and estimates 
for linear elliptic and parabolic partial differential equations. 
The error is measured in a combined norm taking into account both the primal and dual variables. 
We work only on the continuous (often called functional) level 
and do not suppose any specific properties of numerical methods and discretizations.
\end{abstract}

\maketitle
\tableofcontents


\section{Introduction}

The results presented in this paper are functional type a posteriori error equalities and estimates. This type of error control is applicable for any usually (but not necessarily) conforming approximation and involves only global constants,
typically norms of the corresponding inverse operators, i.e., reciprocals of the first non-zero eigenvalues
of the respective differential operators, such as Friedrichs, Poincar\'e, or Maxwell constants. 
The first way of deriving these type error equalities and estimates is the dual variational technique, exposed in detail in the book \cite{NeittaanmakiRepin2004} by Repin and Neittaanm\"aki. In this method the starting point of deriving functionals controlling the error is the dual variational formulation of the problem in question. This method can be applied to a large class of static problems, namely, to convex problems. The second way to arrive at these type equalities and estimates is the method of integral identities, exposed in detail in the book \cite{repinbookone} by Repin (for a more computational point of view see \cite{MaliRepinNeittaanmaki2014} by Mali, Repin, and Neittaanm\"aki). In this method the starting point is the weak formulation of the problem in question. The latter method is easily applicable to linear problems, also including time-dependent problems.

In this paper we expose a way of deriving functional type error equalities and estimates which does not require knowledge of variational or weak formulations of the corresponding problems, and involves only elementary operations. This method was already used in \cite{anjampaulyeq} for static problems with lower order terms, and is extended in the present work to problems without lower order terms, and, more interestingly, to parabolic problems. We note that the method presented here will produce a posteriori error equalities and estimates for mixed approximations, i.e., the error will be measured in a combined norm taking into account both the primal and the dual variable. This is especially useful for mixed methods where one calculates an approximation for both the primal and dual variables, see, e.g., the book \cite{boffibrezzifortinbookone} by Boffi, Brezzi, and Fortin. 
Naturally, we call such an approximation pair a mixed approximation.

We demonstrate our method for the following two elliptic and two parabolic model problems
formulated in a domain $\om\subset\reals^{d}$ and a space-time cylinder $\Xi=I\times\om\subset\reals^{d+1}$,
respectively:
\begin{center}
\begin{tabular}{l@{$\qquad$}l@{$\qquad$}l@{$\,-\,$}l}
	Subsection \ref{ssec:RD} (elliptic) & time-independent Laplace reaction-diffusion & & $\mathring\Delta+1$ \\
	Subsection \ref{ssec:D} (elliptic) & time-independent Laplace diffusion (Poisson equation) & & $\mathring\Delta$ \\
	Subsection \ref{ssec:TRD} (parabolic) & time-dependent Laplace reaction-diffusion & $\p_{\circ}$ & $\mathring\Delta+1$ \\
	Subsection \ref{ssec:H} (parabolic) & time-dependent Laplace diffusion (heat equation) & $\p_{\circ}$ & $\mathring\Delta$ \\
\end{tabular}
\end{center}
Here we denote as usual by $\Delta=\div\na$ the Laplace operator and by $\p_{\circ}$ the time-derivative. 
Note that $\mathring\Delta$ denotes the Dirichlet Laplacian.
The crucial underlying idea is to use the following four isometries for the respective solution operators
mapping to the appropriate Sobolev spaces:
\begin{align*}
u&=(-\mathring\Delta+1)^{-1}f,
&
\normltom{f}^2
&=\norm{\Delta u}^2_{\ltom}
+2\norm{\na u}^2_{\ltom}
+\norm{u}^2_{\ltom}\\
u&=(-\mathring\Delta)^{-1}f,
&
\normltom{f}^2
&=\norm{\Delta u}^2_{\ltom}\\
u&=(\p_{\circ}-\mathring\Delta+1)^{-1}(f,u_{0}),
&
\norm{u_{0}}_{\ltom}^2
+\normltqt{f}^2
+\norm{\na u_{0}}_{\ltom}^2
&=\norm{\p_{\circ}u}^2_{\ltqt}
+\norm{\Delta u}^2_{\ltqt}
+\norm{\na u(T,\,\cdot\,)}_{\ltom}^2\\
&&
&\qquad
+2\norm{\na u}^2_{\ltqt}
+\norm{u}^2_{\ltqt}
+\norm{u(T,\,\cdot\,)}_{\ltom}^2\\
u&=(\p_{\circ}-\mathring\Delta)^{-1}(f,u_{0}),
&
\normltqt{f}^2
+\norm{\na u_{0}}_{\ltom}^2
&=\norm{\p_{\circ}u}^2_{\ltqt}
+\norm{\Delta u}^2_{\ltqt}
+\norm{\na u(T,\,\cdot\,)}_{\ltom}^2
\end{align*}

For simplicity, in this paper we restrict our results to real valued cases,
homogenous Dirichlet boundary conditions, and homogeneous isotropic material coefficients equal to $1$.
Our results can be easily extended to more general cases (complex valued functions and vector fields,
inhomogeneous boundary conditions of different type, inhomogeneous and anisotropic material laws)
and even to a general operator setting considering $\A^{*}\A$ instead of $-\mathring\Delta=-\div\mathring\na$
or $-\Delta=-\mathring\div\na$ for some densely defined and closed linear operator $\A$. 
As an example our results hold for eddy current problems, e.g., for 
$$\p_{\circ}+\text{curl}\,\mathring{\text{curl}}$$ 
in the simplest form. Our method extends also in a nice way to general first order systems, 
which we will treat in a forthcoming contribution.


\section{Time-Independent Elliptic Model Problems: Laplace-Problems} \label{sec:STATIC}

For time-independent problems we denote by $\om\subset\reals^d$, $d\geq1$, an arbitrary (bounded or unbounded) domain, and denote by $\scpltom{\,\cdot\,}{\,\cdot\,}$ resp. $\normltom{\,\cdot\,}$ the inner product resp. norm 
for scalar functions or vector fields in $\ltom$. We define the usual Sobolev spaces
\begin{align*}
	\hoom	& := \set{\varphi\in\ltom}{\na\varphi\in\ltom} , &
	\dom	& := \set{\psi\in\ltom}{\div\psi\in\ltom} , \\
& &	\dzom 	& := \set{\psi\in\dom}{\div\psi=0} ,
\end{align*}
where also the notations $\dom=\mathsf{H}(\div,\om)$ and $\dzom=\mathsf{H}(\div0,\om)$ can be found in the literature.
These are Hilbert spaces equipped with the respective graph norms $\norm{\,\cdot\,}_{\hoom}$, $\norm{\,\cdot\,}_{\dom}$. The space of functions belonging to $\hoom$ and vanishing on the boundary $\p\om$ is defined as the closure of smooth and compactly supported test functions
\begin{equation*}
	\hogaom:=\ol{\cigaom}^{\hoom} ,
\end{equation*}
and hence no regularity assumption on $\om$ resp. $\p\om$ is needed for its definition.
By approximation and continuity we immediately have the rule of partial integration
\begin{equation} \label{eq:partint}
	\forall\,\varphi\in\hogaom
	\quad
	\forall\,\psi\in\dom
	\qquad
	\scpltom{\na\varphi}{\psi}{}=-\scpltom{\varphi}{\div\psi}{} .
\end{equation}
We also define the spaces
\begin{align*}
	\Vom	& := \set{\varphi\in\hoom}{\na\varphi\in\dom} = \set{\varphi\in\hoom}{\Delta\varphi\in\ltom} , \\
	\Vgaom	& := \Vom\cap\hogaom ,
\end{align*}
which are Hilbert spaces with the norm defined by
$$\norm{\cdot}_{\Vom}^2 
:= \norm{\cdot}_{\hoom}^2 + \norm{\na\cdot}_{\dom}^2
=\norm{\cdot}_{\ltom}^2 + 2\norm{\na\cdot}_{\ltom}^2+\norm{\Delta\cdot}_{\ltom}^2.$$ 
If the domain is bounded (at least in one direction, i.e., it lies between two parallel hyperplanes), we have the Friedrichs inequality
\begin{equation} \label{eq:Cf}
	\forall \varphi \in \hogaom \qquad
	\normltom{\varphi} \leq \cf \normltom{\na \varphi} ,
\end{equation}
where $\cf>0$ denotes the Friedrichs constant. We note that generally the exact value of the Friedrichs constant is unknown, but it is easy to estimate it from above, see, e.g., \cite[Poincar\'e's estimate I, p. 25]{leisbook}.


\subsection{Reaction-Diffusion ($-\Delta+1$)}
\label{ssec:RD}

In this section we do not need \eqref{eq:Cf} to hold, i.e., the domain is arbitrary. 
A simple reaction-diffusion problem in mixed form consists 
of finding a scalar potential $u\in\hogaom$ and a flux $p\in\dom$ such that
\begin{equation} \label{eq:rd_pde}
\begin{array}{r@{$\;$}c@{$\;$}l l}
-\na u +	p & = & 0 & \quad \textrm{in } \om , \\
	-\div p + u & = & f & \quad \textrm{in } \om , 
\end{array}
\end{equation}
where the source $f$ belongs to $\ltom$. 
The variables $u$ and $p$ are often called primal and dual variable, respectively. Problem \eqref{eq:rd_pde} has a unique solution 
$u=(-\mathring\Delta+1)^{-1}f$ and $p:=\na u$,
which can also be found by Riesz' representation theorem.

\begin{lem}[isometry]
\label{lem:rd_isometry}
The solution operator
\begin{equation*}
	\sop : \ltom \to \Vgaom ; f \mapsto u
\end{equation*}
related to the problem \eqref{eq:rd_pde} is an isometry, i.e., 
$\norm{u}^2_{\hoom}+\norm{\na u}^2_{\dom}=\normltom{f}^2$
or simply $\norm{\sop}=1$.
\end{lem}

\begin{proof}
We simply compute
\begin{align*}
	\normltom{f}^2 = \normltom{-\Delta u + u}^2
	& = \normltom{\Delta u}^2 + \normltom{u}^2 - 2\scpltom{\Delta u}{u} 
	= \norm{u}^2_{\hoom} + \norm{\na u}^2_{\dom} 
\end{align*}
by $-\scpltom{\Delta u}{u}=\scpltom{\na u}{\na u}$, finishing the proof.
\end{proof}

This directly results in an error equality for sufficiently regular/conforming primal approximations.

\begin{theo}[error equality for very conforming primal approximations] \label{thm:rd_reg}
Let $u,\ut \in \Vgaom$ be the exact solution and an arbitrary approximation of the problem \eqref{eq:rd_pde}, respectively. Then
\begin{equation*}
	\norm{u-\ut}^2_{\hoom} + \norm{\na(u-\ut)}^2_{\dom}
	= \normltom{f - \ut + \Delta\ut}^2 .
\end{equation*}
\end{theo}

\begin{proof}
Since $\ut\in\Vgaom$ is very regular, it is the exact solution of $\ut\in\hogaom$ and
\begin{equation*}
\begin{array}{r@{$\;$}c@{$\;$}l l}
	-\Delta\ut + \ut & =: & \ft & \quad \textrm{in } \om , 
\end{array}	
\end{equation*}
i.e., we have $\sop(\ft) = \ut$. Because $\sop$ is linear we then have $\sop(f-\ft) = u-\ut$, 
and because it is an isometry as well,
we directly have
\begin{equation*}
	\norm{u-\ut}^2_{\hoom} + \norm{\na(u-\ut)}^2_{\dom} = \normltom{f-\ft}^2 .
\end{equation*}
\end{proof}

This error equality has limited applicability due to the high regularity requirement. The next theorem holds for mixed approximations where both the primal and dual variable are conforming.

\begin{theo}[error equality for conforming mixed approximations]
\label{thm:rd}
Let $(u,p),(\ut,\pt) \in \hogaom\times\dom$ be the exact solution pair and an arbitrary approximation pair of the problem \eqref{eq:rd_pde}, respectively. Then
\begin{equation*}
	\norm{u-\ut}_{\hoom}^2 + \norm{p-\pt}_{\dom}^2
	= \normltom{f - \ut + \div\pt}^2 + \normltom{\pt - \na\ut}^2 .
\end{equation*}
\end{theo}

\begin{proof}
By the second equation of \eqref{eq:rd_pde} and inserting $0=\na u - p$ we obtain
\begin{align*}
	\normltom{f-\ut+\div\pt}^2 + \normltom{\pt-\na\ut}^2
	& = \normltom{u-\ut + \div(\pt-p)}^2 + \normltom{\pt-p + \na(u-\ut)}^2 \\
	& = \normltom{u-\ut}^2 + \normltom{\div(\pt-p)}^2
		+ 2\scpltom{u-\ut}{\div(\pt-p)} \\
	& \qquad + \normltom{\pt-p}^2 + \normltom{\na(u-\ut)}^2
		+ 2\scpltom{\pt-p}{\na(u-\ut)} \\
	& = \norm{u-\ut}_{\hoom}^2 + \norm{p-\pt}_{\dom}^2 ,
\end{align*}
since by \eqref{eq:partint} the two cross terms cancel each other.
\end{proof}

\begin{rem}
\label{remthm:rd}
We note:
\begin{itemize}
\item[\bf(i)]
The error equality of Theorem \ref{thm:rd_reg} 
is a special case of Theorem \ref{thm:rd} with $\ut\in\Vgaom$ and $\pt=\na\ut$.
\item[\bf(ii)]
Looking at the proof, in Theorem \ref{thm:rd} is it sufficient that
the difference $u-\ut$ satisfies the boundary condition, i.e.,
$(\ut,\pt) \in \hoom\times\dom$ with $u-\ut\in\hogaom$,
which immediately extends the results for problems with inhomogeneous Dirichlet boundary conditions
as long as the approximation $\ut$ satisfies the boundary condition exactly.
\end{itemize}
\end{rem}

\begin{rem}[upper bounds for conforming primal approximations]
\label{rem2thm:rd}
Let $\ut\in\hogaom$. Choosing $\pt=\phi\in\dom$ in the latter theorem we obtain
$$\norm{u-\ut}_{\hoom}^2
\leq	\norm{u-\ut}_{\hoom}^2+\norm{p-\phi}_{\dom}^2
=\normltom{f-\ut+\div\phi}^2+\normltom{\phi-\na\ut}^2,$$
which - without the term in the middle - is the well known conforming error estimate for the reaction-diffusion problem,
see, e.g., \cite[(4.2.11)]{repinbookone}. Note that for $\pt\in\dom$ and choosing $\ut=\varphi\in\hogaom$ 
in the latter theorem we obtain
$$	\norm{p-\pt}_{\dom}^2
\leq	\norm{u-\varphi}_{\hoom}^2+\norm{p-\pt}_{\dom}^2	
=\normltom{f-\varphi+\div\pt}^2+\normltom{\pt-\na\varphi}^2.$$
\end{rem}

The next results involving non-conforming approximations are all simple applications 
of Theorem \ref{thm:rd}. 

\begin{theo}[upper bounds for non-conforming approximations]
\label{thm:rd_nc}
Let $(u,p) \in \hogaom\times\dom$ be the exact solution pair 
and let $(\ut,\pt)\in\ltom\times\ltom$ be an arbitrary approximation pair 
of the problem \eqref{eq:rd_pde}, respectively. Then
\begin{align*}
\textrm{\bf(i)}
&
& 
\normltom{u-\ut}^2
&\leq\big(1+\frac{1}{\gamma}\big) 
\big(\normltom{f-\varphi+\div\phi}^2
+\frac{1}{2}\normltom{\phi-\na\varphi}^2\big)
+\brac{1+\gamma} \normltom{\varphi-\ut}^2,\\
\textrm{\bf(ii)}
&
&
\normltom{p-\pt}^2
&\leq\big(1+\frac{1}{\gamma}\big) 
\big(\frac{1}{2}\normltom{f-\varphi+\div\phi}^2
+\normltom{\phi-\na\varphi}^2\big)
+\brac{1+\gamma}\normltom{\phi-\pt}^2,\\
\textrm{\bf(iii)} 
&
&
\normltom{u-\ut}^2
+\normltom{p-\pt}^2
&\leq\big(1+\frac{1}{\gamma}\big)
\big(\normltom{f-\varphi+\div\phi}^2
+\normltom{\phi-\na\varphi}^2\big)\\
&&	& 
\hspace*{10mm}+\brac{1+\gamma}
\big(\normltom{\varphi-\ut}^2
+\normltom{\phi-\pt}^2\big)
\end{align*}
hold for arbitrary $(\varphi,\phi)\in\hogaom\times\dom$ and every $\gamma>0$.
\end{theo}

\begin{proof}
For all $(\varphi,\phi)\in\ltom\times\dom$ we have
\begin{equation} \label{eq:low1}
	\normltom{f-\varphi+\div\phi}^2
	= \normltom{u-\varphi+\div(\phi-p)}^2
	\leq 2 \big(\normltom{u-\varphi}^2 + \normltom{\div(p-\phi)}^2\big)
\end{equation}
and for all $(\varphi,\phi)\in\hoom\times\ltom$
\begin{equation} \label{eq:low2}
	\normltom{\phi-\na\varphi}^2
	= \normltom{\phi-p+\na(u-\varphi)}^2
	\leq 2\big(\normltom{p-\phi}^2 + \normltom{\na(u-\varphi)}^2\big).
\end{equation}
Applying \eqref{eq:low1} and \eqref{eq:low2} in the form
\begin{align*}
\normltom{u-\varphi}^2 + \normltom{\div(p-\phi)}^2
&\geq\frac{1}{2}	\normltom{f-\varphi+\div\phi}^2,\\
\normltom{p-\phi}^2 + \normltom{\na(u-\varphi)}^2
&\geq\frac{1}{2}	\normltom{\phi-\na\varphi}^2
\end{align*}
to the left hand side of the equation in Theorem \ref{thm:rd} with $(\ut,\pt)=(\varphi,\phi)$ we obtain
\begin{align}
	\normltom{u-\varphi}^2 + 	\normltom{\div(p-\phi)}^2
	& \leq \normltom{f-\varphi+\div\phi}^2
		+ \frac{1}{2} \normltom{\phi-\na\varphi}^2 , \label{eq:nc1} \\
	\normltom{\na(u-\varphi)}^2 + 	\normltom{p-\phi}^2
	& \leq \frac{1}{2} \normltom{f-\varphi+\div\phi}^2
		+ \normltom{\phi-\na\varphi}^2 . \label{eq:nc2}
\end{align}
Parts (i) and (ii) then follow from applying \eqref{eq:nc1} and \eqref{eq:nc2} to
\begin{align*}
	\normltom{u-\ut}^2
	& \leq \big(1+\frac{1}{\gamma}\big) \normltom{u-\varphi}^2
		+ \brac{1+\gamma} \normltom{\varphi-\ut}^2 , \\
	\normltom{p-\pt}^2
	& \leq \big(1+\frac{1}{\gamma}\big) \normltom{p-\phi}^2
		+ \brac{1+\gamma} \normltom{\phi-\pt}^2 ,
\end{align*}
respectively. Part (iii) follows similarly from
\begin{equation*}
	\normltom{u-\ut}^2 + \normltom{p-\pt}^2
	\leq \big(1+\frac{1}{\gamma}\big) \big( \normltom{u-\varphi}^2 +
		\normltom{p-\phi}^2\big)
		+ \brac{1+\gamma} \big( \normltom{\varphi-\ut}^2 + \normltom{\phi-\pt}^2 \big)
\end{equation*}
and estimating the first two norms on the right hand side from above using Theorem \ref{thm:rd} with $(\ut,\pt) = (\varphi,\phi)$.
\end{proof}

\begin{rem}
Parts (i) and (ii) of Theorem \ref{thm:rd_nc} without the factor $1/2$ 
follow directly from (iii) by setting $\pt:=\phi$ and $\ut:=\varphi$.
\end{rem}

Specializing $\varphi:=\ut$ and $\phi:=\pt$,
Theorem \ref{thm:rd_nc} (iii) implies upper bounds for the primal-dual and dual-primal 
mixed approximations we well.

\begin{cor}[upper bounds for semi-conforming approximations]
\label{cor:rd_nc}
\mbox{}
\begin{itemize}
\item[\bf(i)] For $(\ut,\pt)\in\hogaom\times\ltom$ we have
\begin{equation*}
	\normltom{u-\ut}^2 + \normltom{p-\pt}^2
	\leq \big(1+\frac{1}{\gamma}\big) \big(
		\normltom{f-\ut+\div\phi}^2
		+ \normltom{\phi-\na\ut}^2 \big)
		+ \brac{1+\gamma} \normltom{\phi-\pt}^2
\end{equation*}
for an arbitrary $\phi\in\dom$ and $\gamma>0$.
\item[\bf(ii)] For $(\ut,\pt)\in\ltom\times\dom$ we have
\begin{equation*}
	\normltom{u-\ut}^2 + \normltom{p-\pt}^2
	\leq \big(1+\frac{1}{\gamma}\big) \big(
		\normltom{f-\varphi+\div\pt}^2
		+ \normltom{\pt-\na\varphi}^2 \big)
		+ \brac{1+\gamma} \normltom{\varphi-\ut}^2
\end{equation*}
for an arbitrary $\varphi\in\hogaom$ and $\gamma>0$.
\end{itemize}
\end{cor}

The results from Corollary \ref{cor:rd_nc} can be improved:

\begin{cor}[upper and lower bounds for semi-conforming approximations]
\mbox{}
\begin{itemize}
\item[\bf(i)] For $(\ut,\pt)\in\hogaom\times\ltom$ we have
\begin{align*}
	\frac{1}{2} \normltom{\pt-\na\ut}^2
	& \leq \norm{u-\ut}_{\hoom}^2 + \normltom{p-\pt}^2 \\
	& \leq \big(1+\frac{1}{2\gamma}\big) \normltom{f-\ut+\div\phi}^2
		+ \big(1+\frac{1}{\gamma}\big) \normltom{\phi-\na\ut}^2
		+ \brac{1+\gamma} \normltom{\phi-\pt}^2
\end{align*}
for an arbitrary $\phi\in\dom$ and $\gamma>0$.
\item[\bf(ii)] For $(\ut,\pt)\in\ltom\times\dom$ we have
\begin{align*}
	\frac{1}{2} \normltom{f-\ut+\div\pt}^2
	& \leq \normltom{u-\ut}^2 + \norm{p-\pt}_{\dom}^2 \\
	& \leq \big(1+\frac{1}{\gamma}\big) \normltom{f-\varphi+\div\pt}^2
		+ \big(1+\frac{1}{2\gamma}\big) \normltom{\pt-\na\varphi}^2
		+ \brac{1+\gamma} \normltom{\varphi-\ut}^2
\end{align*}
for an arbitrary $\varphi\in\hogaom$ and $\gamma>0$.
\end{itemize}
\end{cor}

\begin{proof}
The lower bounds in (i) and (ii) follow immediately by the triangle inequality,
see \eqref{eq:low1} and \eqref{eq:low2}. 
The upper bound in (i) can be proved by introducing an arbitrary $\phi\in\dom$ 
and using the triangle inequality such that
\begin{equation*}
	\norm{u-\ut}_{\hoom}^2 + \normltom{p-\pt}^2
	\leq \norm{u-\ut}_{\hoom}^2
	+ \big(1+\frac{1}{\gamma}\big) \normltom{p-\phi}^2
	+ \brac{1+\gamma} \normltom{\phi-\pt}^2 .
\end{equation*}
The error equality of Theorem \ref{thm:rd} for the pair $(\ut,\phi)\in\hogaom\times\dom$ yields
\begin{equation*}
	\norm{u-\ut}_{\hoom}^2 + \normltom{p-\phi}^2
	\leq \normltom{f-\ut+\div\phi}^2 + \normltom{\phi-\na\ut}^2 
\end{equation*}
and the remaining unknown term can be estimated by \eqref{eq:nc2}, i.e.,
\begin{equation*}
	\frac{1}{\gamma} \normltom{p-\phi}^2
	\leq \frac{1}{2\gamma} \normltom{f-\varphi+\div\phi}^2
		+ \frac{1}{\gamma} \normltom{\phi-\na\varphi}^2 .
\end{equation*}
The upper bound in (ii) is shown analogously by using Theorem \ref{thm:rd} and the estimate \eqref{eq:nc1}.
\end{proof}


\subsection{Laplace ($-\Delta$)}
\label{ssec:D}

For this section we assume that the domain $\om$ 
is bounded at least in one direction so that the Friedrichs inequality \eqref{eq:Cf} holds. 
The Poisson problem 
in the mixed form consists of finding a scalar potential $u\in\hogaom$ and a flux $p\in\dom$ such that
\begin{equation} \label{eq:d_pde}
\begin{array}{r@{$\;$}c@{$\;$}l l}
-\na u + 	p & = & 0 & \quad \textrm{in } \om , \\
	-\div p & = & f & \quad \textrm{in } \om , 
\end{array}
\end{equation}
where the source $f$ belongs to $\ltom$. 
Again, the variables $u$ and $p$ are often called primal and dual variable, respectively. 
As before, problem \eqref{eq:d_pde} has a unique solution 
$u=(-\mathring\Delta)^{-1}f$ and $p:=\na u$,
which can also be found by Riesz' representation theorem.
We define an alternative norm on $\Vgaom$ by $\normltom{\Delta\cdot}$.
Note that this is indeed a norm: By \eqref{eq:Cf} we have 
\begin{align} \label{eq:normequiv}
	& \normltom{\na\varphi}^2
	= - \scpltom{\varphi}{\Delta\varphi}
	\leq \normltom{\varphi} \normltom{\Delta\varphi}
	\leq \cf \normltom{\na\varphi} \normltom{\Delta\varphi}
\end{align}
and hence
\begin{align}
\label{eq:Cftwo}
\forall\,\varphi\in\Vgaom\qquad
\normltom{\na\varphi}\leq\cf\normltom{\Delta\varphi}.
\end{align}
Together with \eqref{eq:Cf} the equivalence of $\norm{\cdot}_{\Vom}$ and $\normltom{\Delta\cdot}$ 
on $\Vgaom$ follows, i.e.,
\begin{align*}
\forall\,\varphi\in\Vgaom\qquad
\normltom{\Delta\varphi}^2
\leq\norm{\varphi}_{\Vom}^2
=\normltom{\varphi}^2
+2\normltom{\na\varphi}^2
+\normltom{\Delta\varphi}^2
\leq(1+\cf^2)^2\normltom{\Delta\varphi}^2.
\end{align*}
We also emphasize that \eqref{eq:Cftwo} immediately implies the 
following divergence estimate for irrotational vector fields
\begin{align}
\label{eq:Cfthree}
\forall\,\phi\in\na\hogaom\cap\dom\qquad
\normltom{\phi}\leq\cf\normltom{\div\phi}.
\end{align}

\begin{lem}[isometry]
\label{lem:d_isometry}
The solution operator
\begin{equation*}
	\sop : \ltom \to \big( \Vgaom,\normltom{\Delta\cdot} \big) ; f \mapsto u
\end{equation*}
related to the problem \eqref{eq:d_pde} is an isometry, i.e.,
$\normltom{\Delta u}=\normltom{f}$
or simply $\norm{\sop}=1$.
\end{lem}

\begin{proof}
Note $\normltom{f} = \normltom{\Delta u}$.
\end{proof}

Similarly to Theorem \ref{thm:rd_reg} we immediately get the following.

\begin{theo}[error equality for very conforming primal approximations]
\label{thm:d_reg}
Let $u,\ut \in \Vgaom$ be the exact solution and an arbitrary approximation of the problem \eqref{eq:d_pde}, respectively. Then
\begin{equation*}
	\normltom{\Delta(u-\ut)}
	= \normltom{f + \Delta\ut} .
\end{equation*}
\end{theo}

Note that the Friedrichs constant is absent from the above error equality.


Next we consider conforming mixed approximations for which a two-sided error estimate can be derived by using the error equality for the reaction-diffusion problem from Theorem \ref{thm:rd}.

\begin{theo}[upper and lower bounds for conforming mixed approximations]
\label{thm:d}
Assume $(u,p),(\ut,\pt)$ in $\hogaom\times\dom$ 
to be the exact solution and an arbitrary approximation of the problem \eqref{eq:d_pde}, respectively. Then
\begin{align*}
&\qquad	\max\Big\{
		\normltom{f+\div\pt}^2 + \frac{1}{2} \normltom{\pt-\na\ut}^2
		\;;\;
		\frac{1}{1+\cf^2} \normltom{\pt-\na\ut}^2
	\Big\} 	\\
&	\leq \normltom{\na(u-\ut)}^2 + \norm{p-\pt}_{\dom}^2
	\leq \big(1+4\cf^2\big) \normltom{f+\div\pt}^2 + 2\normltom{\pt-\na\ut}^2 .
\end{align*}
\end{theo}

\begin{proof}
By writing the second equation of \eqref{eq:d_pde} as
\begin{equation*}
	-\div p + u = f + u,
\end{equation*}
$u$ solves the reaction-diffusion problem \eqref{eq:rd_pde} 
with right hand side $f+u$. Theorem \ref{thm:rd} gives then
\begin{equation} \label{eq:d1}
	\norm{u-\ut}_{\hoom}^2 + \norm{p-\pt}_{\dom}^2
	= \normltom{f + u - \ut + \div\pt}^2 + \normltom{\pt-\na\ut}^2 .
\end{equation}
The rest of the proof concentrates on removing the exact solution from the right hand side. By using \eqref{eq:Cf} we estimate the first term on the right hand side as
\begin{align}
	\normltom{f + u - \ut + \div\pt}^2
	& = \normltom{f+\div\pt}^2 + \normltom{u-\ut}^2 + 2\scpltom{f+\div\pt}{u-\ut} \nonumber \\
	& \leq \normltom{f+\div\pt}^2 + \normltom{u-\ut}^2 + 2\cf \normltom{f+\div\pt} \normltom{\na(u-\ut)} \label{eq:d2}  \\
	& \leq \normltom{f+\div\pt}^2 + \normltom{u-\ut}^2 + \gamma \cf^2 \normltom{f+\div\pt}^2
		+  \gamma^{-1} \normltom{\na(u-\ut)}^2 ,\nonumber
\end{align}
which holds for any $\gamma>0$. By choosing $\gamma=2$ 
(there is no need to over-estimate by fixing $\gamma$, but we do it here for simplicity), utilizing 
\begin{align}
\label{fdivformula}
\normltom{f+\div\pt}=\normltom{\div(p-\pt)},
\end{align} 
and combining \eqref{eq:d1} and \eqref{eq:d2} we obtain
\begin{align*}
\frac{1}{2}\normltom{\na(u-\ut)}^2 + \norm{p-\pt}_{\ltom}^2
	&\leq 2\cf^2 \normltom{f+\div\pt}^2 + \normltom{\pt-\na\ut}^2.
\end{align*}
Thus
$\normltom{\na(u-\ut)}^2 + \norm{p-\pt}_{\ltom}^2
\leq 4\cf^2 \normltom{f+\div\pt}^2 + 2\normltom{\pt-\na\ut}^2$
and adding \eqref{fdivformula} to both sides shows the upper bound.
The first lower bound follows from \eqref{eq:low2} (with $\varphi=\ut$ and $\phi=\pt$)
and adding \eqref{fdivformula} to both sides.
By estimating $\norm{\cdot}_{\hoom}^2 \leq (1+\cf^2) \normltom{\na\cdot}^2$ in \eqref{eq:d1} we obtain
\begin{equation*}
	\normltom{\pt-\na\ut}^2
	\leq
	(1+\cf^2) \normltom{\na(u-\ut)}^2 + \norm{p-\pt}_{\dom}^2,
\end{equation*}
which gives the second lower bound. 
\end{proof}

Note that $\normltom{\na \cdot}$ is equivalent to the full $\hoom$ norm due to \eqref{eq:Cf}, thus providing appropriate error control.

For deriving estimates for non-conforming approximations, we use the weak formulation of the problem as the starting point, i.e., we essentially use the integral identity technique exposed in \cite{repinbookone}. This is the only part in this paper where we use this technique. In the forthcoming proof we utilize the $\ltom$-orthogonal Helmholtz decomposition
\begin{align}
\label{helmholtz}
	\ltom = \na\hogaom \oplus \dzom ,
\end{align}
where $\na\hogaom$ is a closed subspace of $\ltom$ 
due to the Friedrichs inequality \eqref{eq:Cf}. 
We denote by $\pi_{\na}$ and $\pi_{0}$ the orthogonal (Helmholtz) projectors 
onto $\na\hogaom$ and $\dzom$, respectively.

\begin{theo}[upper bounds for non-conforming approximations]
\label{thm:d_nc}
Let the pairs $(u,p)\in\hogaom\times\dom$ and $(\ut,\pt)\in\ltom\times\ltom$ be the exact solution 
and an arbitrary approximation of problem \eqref{eq:d_pde}, respectively. Then
\begin{equation*}
\begin{array}{r@{$\quad$}l}
	\textrm{\bf(i)} & \displaystyle
	\normltom{u-\ut}
	\leq \cf^2\normltom{f+\div\phi}
		+ \cf\normltom{\phi-\na\varphi} 
		+ \normltom{\varphi-\ut},
\\[1em]
	\textrm{\bf(ii)} &
	\normltom{p-\pt}^2
	\leq \big( \cf\normltom{f+\div\phi}
		+ \normltom{\phi-\pt} \big)^2
		+ \normltom{\pt-\na\varphi}^2 
\end{array}	
\end{equation*}
hold for arbitrary $(\varphi,\phi)\in\hogaom\times\dom$.
\end{theo}

\begin{rem}[upper bounds for conforming approximations]
\label{remthm:d_nc}
Let $\ut\in\hogaom$. Choosing $\pt:=\na\ut$ and $\varphi:=\ut$ 
in Theorem \ref{thm:d_nc} (ii) we obtain for all $\phi\in\dom$
$$	\normltom{\na(u-\ut)}
\leq\cf\normltom{f+\div\phi}+\normltom{\phi-\na\ut},$$
which is the well known conforming error estimate for the Poisson problem,
see, e.g., \cite[Theorem 3.3]{repinbookone}.
\end{rem}

\begin{proof}[Proof of Theorem \ref{thm:d_nc}]
To prove (ii) we decompose the error
\begin{equation} \label{eq:d_nc1}
	\normltom{p-\pt}^2
	= \normltom{\pi_{\na}(p-\pt)}^2 + \normltom{\pi_{0}\pt}^2 
	= \normltom{\na u-\pi_{\na}\pt}^2 + \normltom{\pi_{0}\pt}^2 
\end{equation}
using the Helmholtz decomposition.
The second part of \eqref{eq:d_nc1} can be estimated using orthogonality, i.e.,
\begin{equation*}
	\normltom{\pi_{0}\pt}^2
	= \scpltom{\pi_{0}\pt}{\pt-\na\varphi}
	\leq \normltom{\pi_{0}\pt}\normltom{\pt-\na\varphi} 
\end{equation*}
for all $\varphi\in\hogaom$, which yields the estimate
\begin{equation} \label{eq:d_nc2}
	\forall\,\varphi\in\hogaom
	\qquad
	\normltom{\pi_{0}\pt} \leq \normltom{\pt-\na\varphi} .
\end{equation}
For estimating the first part of \eqref{eq:d_nc1} we utilize the weak formulation,
i.e., $u$ satisfies
\begin{equation*}
	\forall\,\psi\in\hogaom
	\qquad
	\scpltom{\na u}{\na \psi} = \scpltom{f}{\psi} .
\end{equation*}
Substracting $\scpltom{\pt}{\na \psi}$ from both sides, 
introducing an arbitrary $\phi\in\dom$, and using \eqref{eq:partint} we obtain
\begin{align*} 
	\scpltom{\na u-\pt}{\na \psi}
	& = \scpltom{f}{\psi} - \scpltom{\pt}{\na \psi} 
    = \scpltom{f+\div\phi}{\psi} + \scpltom{\phi-\pt}{\na \psi}  \\
	& \leq \big( \cf\normltom{f+\div\phi}
		+ \normltom{\phi-\pt} \big) \normltom{\na \psi} 
\end{align*}
for all $\psi\in\hogaom$,
where we have used also \eqref{eq:Cf}. Hence by orthogonality
\begin{align*} 
\begin{split}
	\scpltom{\na u-\pi_{\na}\pt}{\na \psi}
=	\scpltom{\na u-\pt}{\na \psi}
 \leq \big( \cf\normltom{f+\div\phi}
		+ \normltom{\phi-\pt} \big) \normltom{\na \psi} 
\end{split}
\end{align*}
for all $\psi\in\hogaom$. Especially for $\na\psi:=\na u-\pi_{\na}\pt\in\na\hogaom$ 
we get the error estimate
\begin{align} 
\label{eq:d_nc3}
	\forall\,\phi\in\dom
	\qquad
	\normltom{\na u-\pi_{\na}\pt}
	\leq \cf\normltom{f+\div\phi}
		+ \normltom{\phi-\pt}
\end{align}
for the first part of the error \eqref{eq:d_nc1}. 
Combining \eqref{eq:d_nc2} and \eqref{eq:d_nc3} with \eqref{eq:d_nc1} results in (ii).
The upper bound in (i) is seen by introducing an arbitrary $\varphi\in\hogaom$ and estimating
\begin{equation*}
	\normltom{u-\ut}
	\leq \normltom{u-\varphi} + \normltom{\varphi-\ut}
	\leq \cf \normltom{\na(u-\varphi)} + \normltom{\varphi-\ut},
\end{equation*}
where we have used \eqref{eq:Cf}. The proof is complete after estimating the first term on the right hand side by (ii) with $\pt = \na\varphi$.
\end{proof}

Results for semi-conforming mixed approximations 
follow from combining the estimates of Theorem \ref{thm:d_nc}.

\begin{cor}[Upper and lower bounds for semi-conforming approximations]
\label{cor:d_nc_mixed}
\mbox{}
\begin{itemize}
\item[\bf(i)] 
Let $(\ut,\pt)\in\hogaom\times\ltom$. Then
\begin{align*}
\frac{1}{2}\normltom{\pt-\na\ut}^2
&\leq\normltom{\na(u-\ut)}^2 
+\normltom{p-\pt}^2\\
&\leq\big(\cf\normltom{f+\div\theta}
+\normltom{\theta-\na\ut}\big)^2\\
&\qquad+\big(\cf\normltom{f+\div\phi} 
+\normltom{\phi-\pt}\big)^2
+\normltom{\pt-\na\varphi}^2 
\end{align*}
holds for all $\theta,\phi\in\dom$ and all $\varphi\in\hogaom$.
Especially one can choose $\varphi=\ut$.
\item[\bf(ii)] 
Let $(\ut,\pt)\in\ltom\times\dom$. Then 
\begin{align*}
&\qquad\normltom{u-\ut}^2 
+\norm{p-\pt}_{\dom}^2\\
&\leq\big(\cf^2\normltom{f+\div\theta}
+\cf\normltom{\theta-\na\psi} 
+\normltom{\psi-\ut}\big)^2
+\big(\cf\normltom{f+\div\phi}
+\normltom{\phi-\pt}\big)^2\\
&\qquad+\normltom{\pt-\na\varphi}^2 
+\normltom{f+\div\pt}^2 
\end{align*}
holds for all $\theta,\phi\in\dom$ and all $\psi,\varphi\in\hogaom$.
Especially, for $\phi=\pt$ it holds
\begin{align*}
&\qquad\normltom{u-\ut}^2 
+\norm{p-\pt}_{\dom}^2\\
&\leq\big(\cf^2\normltom{f+\div\theta}
+\cf\normltom{\theta-\na\psi} 
+\normltom{\psi-\ut}\big)^2
+(\cf^2+1)\normltom{f+\div\pt}^2
+\normltom{\pt-\na\varphi}^2 
\end{align*}
for all $\theta\in\dom$ and all $\psi,\varphi\in\hogaom$.
Especially one can choose $\theta=\pt$.
\end{itemize}
\end{cor}

\begin{proof}
For $(\ut,\pt)\in\hogaom\times\ltom$ Theorem \ref{thm:d_nc} (ii) (applied twice) shows
\begin{align*}
&\qquad\normltom{\na(u-\ut)}^2 
+\normltom{p-\pt}^2\\
&\leq\big(\cf\normltom{f+\div\theta}
+\normltom{\theta-\na\ut}\big)^2
+\big(\cf\normltom{f+\div\phi} 
+\normltom{\phi-\pt}\big)^2\\
&\qquad+\normltom{\na\ut-\na\psi}^2
+\normltom{\pt-\na\varphi}^2
\end{align*}
for all $\theta,\phi\in\dom$ and all $\psi,\varphi\in\hogaom$.
Choosing $\psi:=\ut$ proves (i).
For $(\ut,\pt)\in\ltom\times\dom$ Theorem \ref{thm:d_nc} and \eqref{fdivformula} yield (ii).
\end{proof}


\section{Time-Dependent Parabolic Model Problems: Heat-Problems} \label{sec:TIME}

For time-dependent problems we extend the notation of the previous paragraph to the space-time domain 
$\qt:=I\times\om$, where $\Omega$ is as before and $I:=(0,T)$ is a time-interval with $T>0$. 
Its mantle boundary is $\st$. 
We denote by
\begin{equation*}
	\scpltqt{\,\cdot\,}{\,\cdot\,} 
	= \int_{I} \scpltom{\,\cdot\,}{\,\cdot\,} \, \dlo 
	= \int_{I} \int_{\om} \scp{\,\cdot\,}{\,\cdot\,} \, \dld \, \dlo,
	\qquad 
	\normltqt{\,\cdot\,}^2 = \scpltqt{\,\cdot\,}{\,\cdot\,}
\end{equation*}
the inner product and norm for scalar functions or vector fields in $\ltqt$,
and define the Sobolev spaces
\begin{align*}
	\hzoqt  := \set{\varphi\in\ltqt}{\na\varphi\in\ltqt}, \qquad
	\dqt  := \set{\phi\in\ltqt}{\div\phi\in\ltqt} .
\end{align*}
Here $\na$ and $\div$ are as before the spatial gradient and divergence, respectively. These are Hilbert spaces equipped with the respective graph norms $\norm{\,\cdot\,}_{\hzoqt}$, $\norm{\,\cdot\,}_{\dqt}$. We also define the Sobolev spaces
\begin{equation*}
	\hozqt := \set{\varphi\in\ltqt}{\p_{\circ}\varphi\in\ltqt}   ,\qquad
	\hooqt := \hzoqt \cap \hozqt = \ho(\Xi)  ,	
\end{equation*}
where $\p_{\circ}$ denotes the derivative with respect to time,
which are Hilbert spaces equipped with the respective graph norms as well. Note, e.g.,
$\norm{\,\cdot\,}_{\hooqt}^2 = \normltqt{\,\cdot\,}^2 + \normltqt{\na\,\cdot\,}^2 + \normltqt{\p_{\circ}\,\cdot\,}^2$. 
The corresponding Sobolev spaces of functions vanishing on the mantle boundary are introduced by
\begin{align*}
	\hzostqt  := \set{\varphi\in\hzoqt}{\varphi(t,\,\cdot\,)\in\hogaom \text{ a.e. } t\in I}, \qquad
	\hoostqt  := 
		\hzostqt \cap 	\hozqt.
\end{align*}
We also define the more regular Sobolev spaces
\begin{align*}
	\Wqt  := \set{\varphi\in\hooqt}{\na\varphi\in\dqt}, \qquad
	\Wstqt  := \set{\varphi\in\hoostqt}{\na\varphi\in\dqt} ,
\end{align*}
which are Hilbert spaces with the graph norm 
$\norm{\cdot}_{\Wqt}^2 = \norm{\cdot}_{\hooqt}^2 + \norm{\na\cdot}_{\dqt}^2$.
From \eqref{eq:partint} we directly obtain
\begin{equation} 
\label{eq:tpartint}
	\forall\,\varphi\in\hzostqt
	\quad
	\forall\,\phi\in\dqt
	\qquad
	\scpltqt{\na\varphi}{\phi} = -\scpltqt{\varphi}{\div\phi} ,
\end{equation}
and from \eqref{eq:Cf} we obtain
\begin{equation}
 \label{eq:tCf}
	\forall \varphi \in \hzostqt \qquad
	\normltqt{\varphi} \leq \cf \normltqt{\na \varphi} ,
\end{equation}
if the spatial domain $\Omega$ is bounded at least in one direction.
Moreover, for all (functions or vector fields) $\varphi\in\hozqt$ it holds
\begin{align} 
\label{eq:timecross}
2	\scpltqt{\p_{\circ}\varphi}{\varphi}{}
		= \int_\om \int_{I} \p_{\circ}|\varphi|^{2}\,\dlo\,\dld 
		=   \normltom{\varphi(T,\,\cdot\,)}^2 -
		\normltom{\varphi(0,\,\cdot\,)}^2  . 
\end{align}


\subsection{Reaction-Diffusion ($\p_{\circ}-\Delta+1$)}
\label{ssec:TRD}

In this section we do not need \eqref{eq:tCf} to hold, i.e., the spatial domain $\Omega$ may even be unbounded. The time-dependent reaction-diffusion problem consists of finding a scalar function 
$u\in\hoostqt$ and a flux $p\in\dqt$ such that
\begin{equation} 
\label{eq:trd_pde}
\begin{array}{r@{$\;$}c@{$\;$}l l}
-\na u + 	p & = & 0 & \quad \textrm{in } \qt , \\
	\p_{\circ} u - \div p + u & = & f & \quad \textrm{in } \qt , \\
	u(0,\,\cdot\,) & = & u_0 & \quad \textrm{in } \om ,
\end{array}
\end{equation}
where the source $f$ belongs to $\ltqt$ and the initial value $u_0$ belongs to $\hogaom$. 
Using classical techniques, e.g., those leading 
to \cite[p. 111, (2.15)]{ladybook} and \cite[p. 112, Theorem 2.1]{ladybook},
and computing $\normltqt{(\p_{\circ}-\Delta+1)\varphi}^2$ by applying \eqref{eq:timecross} twice shows 
\begin{align}
\label{reacdiffcomp}
	\normltqt{(\p_{\circ}-\Delta+1)\varphi}^2 + \norm{\varphi(0,\,\cdot\,)}_{\hoom}^2
	= \norm{\varphi}_{\hooqt}^2 + \norm{\na\varphi}_{\dqt}^2
	+ \norm{\varphi(T,\,\cdot\,)}_{\hoom}^2
	=:\|\varphi\|_{\Wstqt}^2 
\end{align}
for all $\varphi\in\Wstqt$. 


\begin{rem}
\label{lem:trd_isometry:rempartiet}
Note that \eqref{reacdiffcomp} holds for smoother functions and thus by approximation also in $\Wstqt$.
Indeed for $\varphi\in\mathsf{X}(\Xi):=\set{\psi\in\Wstqt}{\p_{\circ}\psi\in\hzostqt}$
and $\omega\in\reals$ we can compute
by \eqref{eq:timecross}
\begin{align*}
&\qquad\normltqt{(\p_{\circ}-\Delta+\omega)\varphi}^2\\
&=\normltqt{\p_{\circ}\varphi}^2
+\normltqt{\Delta\varphi}^2
+\omega^2\normltqt{\varphi}^2
-2\scpltqt{\p_{\circ}\varphi}{\Delta\varphi}
+2\omega\scpltqt{\p_{\circ}\varphi}{\varphi}
-2\omega\scpltqt{\Delta\varphi}{\varphi}\\
&=\normltqt{\p_{\circ}\varphi}^2
+\omega^2\normltqt{\varphi}^2
+2\omega\normltqt{\na\varphi}^2
+\normltqt{\Delta\varphi}^2
+2\scpltqt{\p_{\circ}\na\varphi}{\na\varphi}
+2\omega\scpltqt{\p_{\circ}\varphi}{\varphi}\\
&=\normltqt{\p_{\circ}\varphi}^2
+\omega^2\normltqt{\varphi}^2
+2\omega\normltqt{\na\varphi}^2
+\normltqt{\Delta\varphi}^2\\
&\qquad+\normltom{\na\varphi(T,\,\cdot\,)}^2
-\normltom{\na\varphi(0,\,\cdot\,)}^2
+\omega\normltom{\varphi(T,\,\cdot\,)}^2
-\omega\normltom{\varphi(0,\,\cdot\,)}^2.
\end{align*}
For $\omega=1$ we obtain \eqref{reacdiffcomp}, i.e.,
$$\normltqt{(\p_{\circ}-\Delta+1)\varphi}^2
=\norm{\varphi}_{\hooqt}^2
+\norm{\na\varphi}_{\dqt}^2
+\norm{\varphi(T,\,\cdot\,)}_{\hoom}^2
-\norm{\varphi(0,\,\cdot\,)}_{\hoom}^2,$$
for these smoother $\varphi$.
As $\mathsf{X}(\Xi)$ is dense in $\Wstqt$, see, e.g., \cite[p. 111/112]{ladybook},
we obtain \eqref{reacdiffcomp} by approximation.
\end{rem}

The following solution theory holds.

\begin{lem}[isometry]
\label{lem:trd_isometry}
\eqref{eq:trd_pde} is uniquely solvable in $\Wstqt$ and the related solution operator
\begin{equation*}
	\sop : \ltqt\times\hogaom \to 
	\big(\Wstqt,\|\,\cdot\,\|_{\Wstqt}\big) ; (f,u_0) \mapsto u
\end{equation*}
is an isometry, i.e., \eqref{reacdiffcomp} holds for $u=\sop(f,u_{0})$, that is
$$\|u\|_{\Wstqt}^2
=\norm{u}_{\hooqt}^2
+\norm{\na u}_{\dqt}^2
+\norm{u(T,\,\cdot\,)}_{\hoom}^2
=\normltqt{f}^2
+\norm{u_{0}}_{\hoom}^2,$$ 
or simply $\norm{\sop}=1$.
\end{lem}

\begin{proof}
Use \eqref{reacdiffcomp} and the same arguments as in the proof of \cite[p. 112, Theorem 2.1]{ladybook}.
\end{proof}

\begin{theo}[error equality for very conforming primal approximations] 
\label{thm:trd_reg}
Let $u,\ut \in \Wstqt$ be the exact solution and an arbitrary approximation of problem \eqref{eq:trd_pde}, respectively. Then
\begin{equation*}
	\norm{u-\ut}_{\hooqt}^2 + \norm{\na(u-\ut)}_{\dqt}^2
	+ \norm{(u-\ut)(T,\,\cdot\,)}_{\hoom}^2
	= \normltqt{f-\p_{\circ}\ut-\ut+\Delta\ut}^2 + \norm{u_0-\ut(0,\,\cdot\,)}_{\hoom}^2 .
\end{equation*}
\end{theo}

\begin{proof}
The proof is as simple as the proof of Theorem \ref{thm:rd_reg}.
\end{proof}

\begin{theo}[error equality for conforming mixed approximations]
\label{thm:trd}
Let $(u,p),(\ut,\pt) \in \hoostqt\times\dqt$ be the exact solution and an arbitrary approximation of problem \eqref{eq:trd_pde}, respectively. Then
\begin{align*}
&\qquad	\norm{u-\ut}_{\hzoqt}^2 + \normltqt{p-\pt}^2 + \normltqt{\p_{\circ}(u-\ut) + \div(\pt-p)}^2
	+ \normltom{(u-\ut)(T,\,\cdot\,)}^2 \\
&	= \normltqt{f-\p_{\circ}\ut-\ut+\div\pt}^2 + \normltqt{\pt-\na\ut}^2
	+ \normltom{u_0-\ut(0,\,\cdot\,)}^2 .
\end{align*}
\end{theo}

\begin{proof}
By the second equation of \eqref{eq:trd_pde} and using \eqref{eq:timecross} we obtain
\begin{align}
\label{eq:trd1}
\begin{split}
	\normltqt{f-\ut+\div\pt-\p_{\circ}\ut}^2
	& = \normltqt{u-\ut + \p_{\circ}(u-\ut) + \div(\pt-p)}^2  \\
	& = \normltqt{u-\ut}^2 + \normltqt{\p_{\circ}(u-\ut) + \div(\pt-p)}^2  \\
	& \qquad  + 2\scpltqt{u-\ut}{\p_{\circ}(u-\ut)} + 2\scpltqt{u-\ut}{\div(\pt-p)}\\
	& = \normltqt{u-\ut}^2 + \normltqt{\p_{\circ}(u-\ut) + \div(\pt-p)}^2  \\
	& \qquad + \normltom{(u-\ut)(T,\,\cdot\,)}^2
			 - \normltom{(u-\ut)(0,\,\cdot\,)}^2
			 + 2\scpltqt{u-\ut}{\div(\pt-p)} .
\end{split}
\end{align}
On the other hand, by inserting the first equation of \eqref{eq:trd_pde} we obtain
\begin{align}
\label{eq:trd2} 
	\normltqt{\pt-\na\ut}^2
 = \normltqt{\pt-p + \na(u-\ut)}^2 
 = \normltqt{\pt-p}^2 + \normltqt{\na(u-\ut)}^2
		+ 2\scpltqt{\pt-p}{\na(u-\ut)} . 
\end{align}
By \eqref{eq:tpartint} and adding \eqref{eq:trd1} and \eqref{eq:trd2} together shows the assertion.
\end{proof}


\subsection{Heat Equation ($\p_{\circ}-\Delta$)}
\label{ssec:H}

In this section we assume that the Friedrichs inequality \eqref{eq:tCf} holds. 
The heat equation consists of finding a scalar function $u\in\hoostqt$ and a flux $p\in\dqt$ such that
\begin{equation} 
\label{eq:h_pde}
\begin{array}{r@{$\;$}c@{$\;$}l l}
-\na u +	p & = & 0 & \quad \textrm{in } \qt , \\
	\p_{\circ} u - \div p & = & f & \quad \textrm{in } \qt , \\
	u(0,\,\cdot\,) & = & u_0 & \quad \textrm{in } \om ,
\end{array}
\end{equation}
where again the source $f$ belongs to $\ltqt$ and the initial value $u_0$ belongs to $\hogaom$. 
By \cite[p. 111, (2.15)]{ladybook} we have
\begin{align}
\label{heatcomp}
	\normltqt{(\p_{\circ}-\Delta)\varphi}^2 + \normltom{\na\varphi(0,\,\cdot\,)}^2 
	= \normltqt{\p_{\circ}\varphi}^2 + \normltqt{\Delta\varphi}^2
	+ \normltom{\na\varphi(T,\,\cdot\,)}^2 
	=:|\!|\!|\varphi|\!|\!|_{\Wstqt}^2 
\end{align}
for all $\varphi\in\Wstqt$, see also Remark \ref{lem:trd_isometry:rempartiet} for $\omega=0$. 
\cite[p. 112, Theorem 2.1]{ladybook} shows that \eqref{eq:h_pde} is uniquely solvable in $\Wstqt$
and that the solution satisfies \eqref{heatcomp}.
Note that $|\!|\!|\,\cdot\,|\!|\!|_{\Wstqt}$ is indeed a norm on $\Wstqt$,
which follows from \eqref{eq:Cftwo} and the surrounding caculations. 

\begin{lem}[isometry]
\label{lem:h_isometry}
The solution operator
\begin{equation*}
	\sop : \ltqt\times\hogaom \to \big(\Wstqt,|\!|\!|\,\cdot\,|\!|\!|_{\Wstqt}\big) ; (f,u_0) \mapsto u
\end{equation*}
related to the problem \eqref{eq:h_pde} is an isometry, i.e., 
\eqref{heatcomp} holds for $u=\sop(f,u_{0})$, that is
$$	|\!|\!|u|\!|\!|_{\Wstqt}^2 
=\normltqt{\p_{\circ}u}^2
+\normltqt{\Delta u}^2
+\normltom{\na u(T,\,\cdot\,)}^2
=\normltqt{f}^2
+\normltom{\na u_{0}}^2,$$
or simply $\norm{\sop}=1$.
\end{lem}

\begin{theo}[error equality for very conforming primal approximations] 
\label{thm:h_reg}
Let $u,\ut \in \Wstqt$ be the exact solution and an arbitrary approximation of problem \eqref{eq:h_pde}, respectively. Then
\begin{equation*}
	\normltqt{\p_{\circ}(u-\ut)}^2 + \normltqt{\Delta(u-\ut)}^2
	+ \normltom{\na(u-\ut)(T,\,\cdot\,)}^2
	= \normltqt{f+\Delta\ut-\p_{\circ}\ut}^2 + \normltom{\na(u_0-\ut(0,\,\cdot\,))}^2 .
\end{equation*}
\end{theo}

\begin{proof}
Again, the proof is as simple as the proofs of Theorem \ref{thm:rd_reg} and Theorem \ref{thm:trd_reg}.
\end{proof}

The latter result is similar to the error equality for the reaction-diffusion equation 
in Theorem \ref{thm:d_reg}. Note that the Friedrichs constant is absent.
Next we consider conforming mixed approximations for which a two-sided error estimate can be derived by using the error equality for the time-dependent reaction-diffusion problem from Theorem \ref{thm:trd} 
(similarly to what was done in Theorem \ref{thm:d} for the diffusion problem).

\begin{theo}[error estimate for conforming mixed approximations]
\label{thm:h}
Let $(u,p),(\ut,\pt) \in \hoostqt\times\dqt$ be the exact solution and an arbitrary approximation of problem \eqref{eq:h_pde}, respectively. Then
\begin{align*}
	& \qquad \max\Big\{
		\normltqt{f+\div\pt-\p_{\circ}\ut}^2 + \frac{1}{2}\normltqt{\pt-\na\ut}^2 
		\;;\;
		\frac{1}{1+\cf^2} \big(\normltqt{\pt-\na\ut}^2 + \normltom{u_0-\ut(0,\,\cdot\,)}^2\big)
	  \Big\}	
	\\	
	& \leq \normltqt{\na(u-\ut)}^2 + \normltqt{p-\pt}^2
		+ \normltqt{\p_{\circ}(u-\ut) + \div(\pt-p)}^2 + \normltom{(u-\ut)(T,\,\cdot\,)}^2 \\
	& \leq \brac{1+4\cf^2}\normltqt{f + \div\pt -\p_{\circ}\ut}^2
		+ 2\normltqt{\pt-\na\ut}^2 + 2\normltom{u_0-\ut(0,\,\cdot\,)}^2 .
\end{align*}
\end{theo}

We note that $\normltqt{\na\,\cdot\,}$ is equivalent to the full $\hzoqt$ norm on  $\hzostqt$ 
due to \eqref{eq:tCf} and thus provides appropriate error control.

\begin{proof}
By writing the second equation of \eqref{eq:h_pde} as
\begin{equation*}
	\p_{\circ} u - \div p + u = f + u,
\end{equation*}
$u$ solves the time-dependent reaction-diffusion problem \eqref{eq:trd_pde} with right hand side $f+u$. 
Theorem \ref{thm:trd} shows
\begin{align} 
\label{eq:h1}
\begin{split}
&\qquad	\norm{u-\ut}_{\hzoqt}^2 + \normltqt{p-\pt}^2 + \normltqt{\p_{\circ}(u-\ut) + \div(\pt-p)}^2
	+ \normltom{(u-\ut)(T,\,\cdot\,)}^2 \\
&	= \normltqt{f+u-\ut+\div\pt-\p_{\circ}\ut}^2 + \normltqt{\pt-\na\ut}^2
	+ \normltom{u_0-\ut(0,\,\cdot\,)}^2 .
\end{split}
\end{align}
By using \eqref{eq:tCf} we estimate the first term on the right hand side as
\begin{align}
 \label{eq:h2}
 \begin{split}
	& \qquad \normltqt{f+u-\ut+\div\pt-\p_{\circ}\ut}^2 \\
	& = \normltqt{f+\div\pt-\p_{\circ}\ut}^2 + \normltqt{u-\ut}^2
		+ 2\scpltqt{f+\div\pt-\p_{\circ}\ut}{u-\ut} \\
	& \leq \normltqt{f+\div\pt-\p_{\circ}\ut}^2 + \normltqt{u-\ut}^2
		+ 2\cf\normltqt{f+\div\pt-\p_{\circ}\ut} \normltqt{\na(u-\ut)} \\
	& \leq \normltqt{f+\div\pt-\p_{\circ}\ut}^2 + \normltqt{u-\ut}^2
		+ \gamma\cf^2\normltqt{f+\div\pt-\p_{\circ}\ut}^2
		+ \gamma^{-1} \normltqt{\na(u-\ut)}^2  ,
\end{split}
\end{align}
which holds for any $\gamma>0$. By choosing $\gamma=2$ (there is no need to over-estimate by fixing $\gamma$, but we do it here for brevity) and combining \eqref{eq:h1} and \eqref{eq:h2} we obtain
\begin{align*}
&\qquad	\frac{1}{2} \normltqt{\na(u-\ut)}^2 + \normltqt{p-\pt}^2
		+ \normltqt{\p_{\circ}(u-\ut) + \div(\pt-p)}^2 + \normltom{(u-\ut)(T,\,\cdot\,)}^2 \\
&	\leq \brac{1+2\cf^2} \normltqt{f+\div\pt-\p_{\circ}\ut}^2 + \normltqt{\pt-\na\ut}^2
		+ \normltom{u_0-\ut(0,\,\cdot\,)}^2 .
\end{align*}
Since 
\begin{align}
\label{fdivptformula}
\normltqt{f+\div\pt-\p_{\circ}\ut}=\normltqt{\p_{\circ}(u-\ut) + \div(\pt-p)},
\end{align}
we have proven the upper bound.
To prove the first lower bound we simply insert $0=\na u - p$ in
\begin{equation*}
	\normltqt{\pt-\na\ut}^2
	= \normltqt{\pt-p+\na u-\na\ut}^2
	\leq 2 \brac{\normltqt{\pt-p}^2 + \normltqt{\na(u-\ut)}^2} 
\end{equation*}
and use \eqref{fdivptformula}.
By estimating $\norm{\,\cdot\,}_{\hzoqt}^2 \leq (1+\cf^2) \normltqt{\na\,\cdot\,}^2$ in \eqref{eq:h1} we obtain
\begin{align*}
& \qquad	\normltqt{\pt-\na\ut}^2 + \normltom{u_0-\ut(0,\,\cdot\,)}^2 \\ 
& \leq	(1+\cf^2) \normltqt{\na(u-\ut)}^2 + \normltqt{p-\pt}^2
	+ \normltqt{\p_{\circ}(u-\ut) + \div(\pt-p)}^2
	+ \normltom{(u-\ut)(T,\,\cdot\,)}^2,
\end{align*}
which shows the second lower bound. 
\end{proof}


\bibliographystyle{plain} 
\bibliography{biblio}

\end{document}